\documentclass[a4paper]{amsart}
\usepackage[utf8]{inputenc}
\usepackage[english]{babel}
\usepackage{csquotes}
\usepackage[shortcuts]{extdash}

\usepackage{amsmath}
\usepackage{amsfonts}
\usepackage{amsthm}
\usepackage{amssymb}
\usepackage{mathrsfs}
\usepackage{mathtools}
\usepackage{stmaryrd}

\usepackage[shortlabels]{enumitem}
\setlist[enumerate]{label=(\alph*)}

\usepackage{tikz}
\usepackage{tikz-cd}

\usepackage{breakurl} 
\usepackage[hidelinks, pdfauthor={Gwyn Bellamy, Gerhard Röhrle, Johannes Schmitt}]{hyperref}

\theoremstyle{plain}
\newtheorem{theorem}{Theorem}[section]

\newtheorem{proposition}[theorem]{Proposition}

\theoremstyle{definition}
\newtheorem{definition}[theorem]{Definition}

\theoremstyle{remark}
\newtheorem{remark}[theorem]{Remark}
\newtheorem{notation}[theorem]{Notation}

\DeclareMathOperator{\GL}{GL}
\DeclareMathOperator{\Hom}{Hom}
\DeclareMathOperator{\id}{id}
\DeclareMathOperator{\Irr}{Irr}
\DeclareMathOperator{\Pic}{Pic}
\DeclareMathOperator{\SL}{SL}

\DeclareMathOperator{\Sp}{Sp}
\DeclareMathOperator{\Spec}{Spec}

\DeclareMathOperator{\U}{U}

\newcommand{\CC}{\mathbb C}
\newcommand{\HH}{\mathbb H}

\newcommand{\QQ}{\mathbb Q}
\newcommand{\RR}{\mathbb R}
\newcommand{\ZZ}{\mathbb Z}

\renewcommand{\epsilon}{\varepsilon}
\renewcommand{\phi}{\varphi}
\renewcommand{\theta}{\vartheta}

\newcommand{\SRA}{\mathtt{H}}

\makeatletter
\patchcmd{\@setauthors}{\footnotesize}{}{}{}
\patchcmd{\@setauthors}{\MakeUppercase}{}{}{}
\patchcmd{\@setaddresses}{\scshape}{}{}{}
\patchcmd{\contentsnamefont}{\scshape}{\bfseries}{}{}
\patchcmd{\abstract}{\scshape}{\bfseries}{}{}
\patchcmd{\section}{\scshape}{\bfseries}{}{}
\patchcmd{\@secnumfont}{\mdseries}{\bfseries}{}{}
\patchcmd{\@captionheadfont}{\scshape}{\bfseries}{}{}
\makeatother

\title[Namikawa--Weyl groups of symplectic quotient singularities]{Namikawa--Weyl groups of\\ symplectic quotient singularities}

\author[G.\ Bellamy]{Gwyn Bellamy}
\address{Gwyn Bellamy, The Mathematics and Statistics Building, University of Glasgow, University Place, Glasgow G12 8QQ, Scotland}
\email{gwyn.bellamy@glasgow.ac.uk}

\author[G.\ Röhrle]{Gerhard Röhrle}
\address{Gerhard Röhrle, Ruhr\-/Universität Bochum, Fakultät für Mathematik, Universitätsstraße 150, 44801 Bochum, Germany}
\email{gerhard.roehrle@ruhr-uni-bochum.de}

\author[J.\ Schmitt]{Johannes Schmitt}
\address{Johannes Schmitt, Ruhr\-/Universität Bochum, Fakultät für Mathematik, Universitätsstraße 150, 44801 Bochum, Germany}
\email{johannes.schmitt@ruhr-uni-bochum.de}



\begin{document}

\begin{abstract}
  We classify the Namikawa--Weyl groups associated to symplectic quotient singularities \(V/G\) when \(G\) is a symplectic reflection group.
  Our classification shows that every irreducible Weyl group can be realized as a factor of the Namikawa--Weyl group of \(V/G\) for a suitable \(G\).
\end{abstract}

\maketitle

\section{Introduction}

Let \(V\) be a finite\-/dimensional complex vector space endowed with a symplectic form and let \(G\leqslant \Sp(V)\) be a finite group. The linear quotient \(V/G \coloneqq \Spec \CC[V]^G\) is a symplectic quotient singularity \cite{Bea00}.
A \emph{\(\QQ\)\=/factorial terminalization} of \(V/G\) is a crepant, projective, birational morphism \(X \to V/G\) such that \(X\) is \(\QQ\)\=/factorial with only terminal singularities. 
Notice that a smooth \(\QQ\)\=/factorial terminalization is exactly a projective symplectic resolution in the sense of \cite{Bea00}.

By a result of Namikawa \cite{Nam15}, a \(\QQ\)\=/factorial terminalization \(\phi:X\to V/G\) is a \emph{relative Mori dream space} over \(V/G\), see \cite{HK00, Oht22}. This means that there are only finitely many \(\QQ\)\=/factorial terminalizations of \(V/G\) up to isomorphism and that the relations between the different isomorphism classes are controlled by a wall\-/and\-/chamber structure in the vector space \(\mathfrak c \coloneqq \Pic(X)\otimes_\ZZ \RR\).
Namikawa's key observation is that the walls in this wall\-/and\-/chamber structure in fact form a hyperplane arrangement in \(\mathfrak c\), commonly called the \emph{Namikawa arrangement}. Namikawa \cite{Nam10} introduced, in addition, a certain real reflection group \(\mathcal W_G \leqslant \mathrm{GL}(\mathfrak{c})\), nowadays called the \emph{Namikawa--Weyl group}, which permutes the hyperplanes in the Namikawa arrangement. The group \(\mathcal W_G\) only depends on \(V/G\) and not on the morphism \(\phi\).

\subsection*{Symplectic reflection groups}

Recall that an element $g \in G$ is a symplectic reflection if the rank of $1 - g$ is 2. The group $G \leqslant \Sp(V)$ is said to be a symplectic reflection group if it is generated by the symplectic reflections it contains. By Verbitsky's Theorem \cite{Ver00}, if $V/G$ admits a symplectic resolution then $G$ must be a symplectic reflection group. We assume for the remainder of the article that $G$ is a symplectic reflection group. 

The purpose of this note is to explain how recent advances in our understanding of symplectic reflection groups  \cite{GRS25,RS26} allow us to explicitly compute the Namikawa--Weyl group for any symplectic reflection group.
There are several reasons why this is important.
Firstly, as explained above, knowing the Namikawa--Weyl group is an essential part of the ongoing programme to classify all $\QQ$-factorial terminalizations of quotients $V/G$ by symplectic reflection groups.
In particular, if one wishes to count the number of $\QQ$-factorial terminalizations admitted by $V/G$ then one needs to know the Namikawa--Weyl group; see \cite{Bel16}.
Secondly, for any given family of conic symplectic singularities, it is a rather subtle problem to decide which Weyl groups can be realised as the Namikawa--Weyl group of a singularity in the family.
For instance, it was shown by Wu \cite[Thm.~1.14]{Wu23} that the Namikawa--Weyl group of a Nakajima quiver variety can never have a factor of type $F_4$.
For quotient singularities, however, it is a consequence of our classification result that every irreducible Weyl group appears as a factor of the Namikawa--Weyl group of \(V/G\) for a suitable \(G\).
More specifically:

\begin{theorem}
  For an irreducible Weyl group \(W\), there is a symplectically irreducible symplectic reflection group \(G\) such that the quotient \(V/G\) has Namikawa--Weyl group of the form \(W\times A_1^k\) for some \(k\in\{0, 1\}\).

  Conversely, if $G$ is a symplectically irreducible symplectic reflection group then the Namikawa--Weyl group of $V/G$ is of the form $W \times A_1^k$, where $W$ is an irreducible Weyl group and $0 \leqslant k \leqslant 5$, except in the case where $G$ is an exceptional complex reflection group of type $G_5, G_7, G_{10}, G_{11}, G_{18}$ or $G_{19}$.
  In these cases, the Namikawa--Weyl group has at least two irreducible factors not of type $A_1$.
\end{theorem}

The theorem follows from Propositions \ref{prop:SL2pairs}, \ref{prop:imprim}, \ref{prop:primimprim} and Tables \ref{tab:compexcep} and \ref{tab:excep}.

For a given irreducible Weyl group $W$, one can ask if there exists a symplectic reflection group whose Namikawa--Weyl group is $W$.
Our classification shows that the answer is yes only for Weyl groups of simply laced type, of type \(B_2\), and of type \(G_2\).
The Weyl groups \(W\) of type \(F_4\) and \(B_m\) for \(m\geqslant 3\) are only realizable as a Namikawa--Weyl group as \(W \times A_1\).

Finally, we note that the Namikawa--Weyl group can be realised as a group of exotic automorphisms of the generalized Calogero--Moser space associated to $G$. See Section~\ref{sec:CM} for details.

\subsection*{Acknowledgements} 

The first author was supported by EPSRC grants UKRI2779 and EP-W013053-1.

\section{The Namikawa--Weyl group}
\label{sec:namikawaweyl}

Throughout, let \(V\) be a finite\-/dimensional complex vector space endowed with a symplectic form and let \(G\leqslant \Sp(V)\) be a symplectic reflection group.

\subsection{Minimal parabolic subgroups}
Bellamy \cite{Bel16} gives a description of the Namikawa--Weyl group of the singularity \(V/G\) via the minimal parabolic subgroups of \(G\). The details of this construction are essential for what follows, so we repeat them here.

\begin{definition}
  A subgroup \(P \leqslant G\) is called \emph{parabolic} if \(P\) is the stabilizer of some vector \(v\in V\).
  A parabolic subgroup \(P\leqslant G\) is called \emph{minimal} if \(\dim(V) - \dim(V^P) = 2\).
  We write \(\mathcal P(G)\) for the set of \(G\)\=/conjugacy classes of minimal parabolic subgroups of \(G\). Since $G$ is assumed to be a symplectic reflection group, it is generated by its minimal parabolic subgroups. 
\end{definition}

For every conjugacy class \(C\in\mathcal P(G)\), we fix a representative \(P\in C\). 
Then, viewed as a representation of \(P\), \(V\) decomposes as $V = V_P \oplus V^P$, where $V_P \coloneqq (V^P)^{\perp}$ is the symplectic complement to $V^P$. Since $P$ is a minimal parabolic, the subspace $V_P$ is a two-dimensional (faithful) representation and hence \(P\) is isomorphic to a finite subgroup of \(\SL_2(\CC) \cong \Sp(V_P)\).
Then the McKay correspondence \cite{McK80,GV83} associates to \(P\) a Weyl group \(W_P\) with real reflection representation \(\mathfrak h_P\).
There is a natural action of the quotient \(\Xi_P = N_G(P) / P\) on \(\mathfrak h_P\) given by a morphism \(\zeta: \Xi_P \to \GL(\mathfrak h_P)\), see below for details.
Note that \(W_P\), \(\mathfrak h_P\) and \(\Xi_P\) only depend on the conjugacy class \(C\) (up to isomorphism), not on the choice of \(P\). Therefore, we write \(W_C\), \(\mathfrak h_C\) and \(\Xi_C\) in the following.
Let \(\mathfrak h_C^{\Xi_C}\) be the fixed space and let \[W_C^{\Xi_C} = \{w \in W_C \mid \zeta(x) w \zeta(x)^{-1} = w\text{ for all }x\in \Xi_C\}\] be the subgroup of elements fixed by $\Xi_C$.

\begin{theorem}[{\cite[Thm.~1.3]{Bel16}}]
  \label{thm:directproduct}
  With the above notation, the Namikawa--Weyl group associated to \(V/G\) is given by the direct product \[\mathcal W_G = \prod_{C\in\mathcal P(G)} W_C^{\Xi_C}\] acting on the space \(\Pic(X)\otimes_\ZZ\RR \cong \bigoplus_{C\in\mathcal P(G)} \mathfrak h_C^{\Xi_C}\) by reflections.
\end{theorem}

\subsection{McKay correspondence}
To construct the Namikawa--Weyl group for a given quotient singularity \(V/G\) one therefore needs to determine the fixed groups \((W_C^{\Xi_C}, \mathfrak h_C^{\Xi_C})\) for each conjugacy class \(C\in \mathcal P(G)\) of minimal parabolic subgroups.
Let \(C\in \mathcal P(G)\) and let \(P\in C\).
We describe the action of \(\Xi_C \cong N_G(P)/P\) on \(\mathfrak h_C\) following \cite[Sect.~2.1]{Bel16}.
Let \(\Irr(P)\) be the set of isomorphism classes of \emph{non\-/trivial} irreducible complex \(P\)\=/modules.
For \(x\in \Xi_C\) and \(\lambda\in \Irr(P)\), we let \({}^x\lambda\) be the element of \(\Irr(P)\) that equals \(\lambda\) as a vector space and on which \(g\in P\) acts via \(g\cdot v = \tilde{x}^{-1} g \tilde{x} \cdot v\) for all \(v\in {}^x\lambda\) and some choice of lift $\tilde{x} \in N_G(P)$ of $x \in \Xi_C$. Up to isomorphism, ${}^x \lambda$ does not depend on the choice of lift. 

The McKay correspondence associates to \(P\) a Dynkin diagram \(\Delta_C\) as follows.
The vertices of \(\Delta_C\) are the elements of \(\Irr(P)\) and two vertices \(\lambda,\lambda'\in \Irr(P)\) are connected by an edge if \(\Hom_P(V_P\otimes_\CC \lambda, \lambda') \neq 0\).
By construction, the action of \(\Xi_C\) on \(\Irr(P)\) induces a diagram automorphism of \(\Delta_C\).

Because \(N_G(P)\) normalizes \(P\), the decomposition \(V = V^P \oplus (V^P)^\perp\) is a decomposition of \(N_G(P)\)-modules. This gives rise to a homomorphism $N_G(P) \to \Sp(V_P) = \SL_2(\CC)$.
For describing the action of \(\Xi_C\) on \(\mathfrak h_C\), only the image of \(N_G(P)\) in \(\SL_2(\CC)\) is relevant. This implies that all possible groups \(W_C^{\Xi_C}\) that can occur in Theorem~\ref{thm:directproduct} must be obtained by first listing all pairs \((H, N)\) with \(H \trianglelefteqslant N \leqslant \SL_2(\CC)\) finite groups and computing the corresponding Weyl group \(W_{H}^{N/H}\).

\begin{table}[h]
  \caption{Kleinian groups \(H\trianglelefteqslant N\leqslant\SL_2(\CC)\), \(H\neq N\), with the Weyl groups \((W_H, \mathfrak h_H)\) and \((W_H^{N/H}, \mathfrak h_H^{N/H})\).}
  \label{tab:kleinianpairs}
  \begin{tabular}{l|l|l|l}
  \(H\)                         & \(N\)       & \(W_H\) & \(W_H^{N/H}\)             \\
  \hline
  \(\mathsf C_2\)               & \(\mathsf C_{2r}\), \(\mathsf D_r\), \(\mathsf T\), \(\mathsf O\), \(\mathsf I\), \((r\geqslant 2)\) & \(A_1\)         & \(A_1\)                                     \\
  \(\mathsf C_n\) (\(n\geqslant 2)\) & \(\mathsf C_{nr}\) (\(r\geqslant 2\))                                                                & \(A_{n - 1}\)   & \(A_{n - 1}\)                               \\
  \(\mathsf C_n\) (\(n\geqslant 2)\) & \(\mathsf D_{r}\) (\(n \mid 2r\))                                                                & \(A_{n - 1}\)   & \(B_m\), \(m = \lfloor \frac{n}{2}\rfloor\) \\
  \(\mathsf D_n\) (\(n\geqslant 2)\) & \(\mathsf D_{2n}\)                                                                              & \(D_{n + 2}\)   & \(B_{n + 1}\)                               \\
  \(\mathsf D_2\)               & \(\mathsf T\), \(\mathsf O\)                                                                    & \(D_4\)         & \(G_2\)                                     \\
  \(\mathsf T\)                 & \(\mathsf O\)                                                                                   & \(E_6\)         & \(F_4\)
  \end{tabular}
\end{table}

\begin{proposition}
  \label{prop:SL2pairs}
  Let \(H, N\leqslant\SL_2(\CC)\) be finite non-trivial groups with \(H\trianglelefteqslant N\) and \(H\neq N\).
  Let \((W_H, \mathfrak h_H)\) be the Weyl group associated to \(H\) via the McKay correspondence.
  Table~\ref{tab:kleinianpairs} lists all possible pairs \((H, N)\) together with the Weyl group \((W_H^{N/H}, \mathfrak h_H^{N/H})\).
\end{proposition}
\begin{proof}
  The pairs \((H, N)\) of finite groups \(H, N\leqslant\SL_2(\CC)\) with \(H\trianglelefteqslant N\) are well-known; see for example \cite[p.~56]{Val64}.

  Determining the groups \(W_H^{N/H}\) is now a case-by-case analysis.
  For \(H = \mathsf C_2\), the action of $N/H$ on $W_H = \mathfrak{S}_2$ is necessarily trivial since the Dynkin diagram is $A_1$, with one vertex and no edges, so every diagram automorphism is trivial. Hence $W_H^{N/H} = W_H = \mathfrak{S}_2$. Likewise, if \(H\) and \(N\) are both cyclic then the conjugation action of $N$ on $H$ is trivial.  This implies that the action of \(N/H\) on \(\mathfrak h_H\) is also trivial.

  Let \(H = \mathsf C_n\) with \(n\geqslant 2\) and \(N = \mathsf D_r\) for some \(r\) with \(n\mid 2r\).
  Let \(h\in H\) be a generator of \(H\) and let \(\zeta_n\in \CC\) be a primitive \(n\)\=/th root of unity.
  Then the Dynkin diagram \(\Delta_H\) associated to \(H\) via the McKay correspondence is of type \(A_{n - 1}\), where the vertices can be labelled by the \(n - 1\) characters \(\chi_k\) defined by \(\chi_k(h) = \zeta_n^k\), with \(k = 1,\dots, n - 1\).
  The vertex \(\chi_k\) is linked to the vertex \(\chi_l\) in \(\Delta_H\) if and only if \(|{l - k}| = 1\).
  Inside \(N\), the element \(h\) is conjugate to \(h^{-1}\).
  Hence the induced action of \(N/H\) on \(\Delta_H\) factors through \(\ZZ/2\ZZ\), where the non-trivial element acts by sending \(\chi_i\) to \(\chi_{n - i}\) for \(1\leqslant i \leqslant m\) with \(m = \lfloor\frac{n}{2}\rfloor\).
  The resulting fixed group \(W_H^{N/H}\) is then again a reflection group \cite[Thm.~32]{Ste68}, \cite{GI14} and of type \(B_m\), see \cite[p.~175, Ex.~(a)]{Ste68}.

  The remaining cases use the same routine arguments, so we omit the details.
\end{proof}

\begin{remark}
	We will abuse notation and often write the Coxeter label of the irreducible factors of $\mathcal{W}_G$ to describe the group. For instance, if $\mathcal{W}_G = \mathfrak{S}_5 \times \mathfrak{S}_7$ then we write simply $\mathcal{W}_G = A_4 \times A_6$. 
\end{remark}

\subsection{Calogero--Moser spaces}\label{sec:CM}

Let $\mathfrak{c}_{\CC} = \mathfrak{c} \otimes_{\RR} \CC$. Etingof and Ginzburg \cite{EG} associate to any symplectic reflection group the \emph{symplectic reflection algebra} $\SRA$ (at $t = 0$), which is a $\CC[\mathfrak{c}_{\CC}]$\=/algebra. For any point $c \in \mathfrak{c}_{\CC}$ one can specialize $\SRA$ to $\SRA_c$. The algebra $\SRA$ is a finite module over its centre and the \emph{generalised Calogero--Moser space} is defined to be $\mathrm{CM}(G) = \mathrm{Spec} \, Z(\SRA)$. This space comes with a morphism to $\mathfrak{c}_{\CC}$ such that the fibre over $c \in \mathfrak{c}_{\CC}$ can be shown to be isomorphic to $\mathrm{CM}_c(G) = \mathrm{Spec} \, Z(\SRA_c)$ using the Satake isomorphism \cite[Thm.~3.1]{EG}. 

The name comes from the fact that when $G = \mathfrak{S}_n$, it was shown in \cite[Thm.~1.13]{EG} that $\mathrm{CM}_c(G)$, for $c \neq 0$, is isomorphic to Wilson's compactified Ca\-lo\-ge\-ro--Moser space \cite{Wilson}. Then \cite[Thm.~1.4]{Bel16} says that:

\begin{proposition}
	The action of $\mathcal{W}_G$ on $\mathfrak{c}_{\CC}$ lifts to an action on $\mathrm{CM}(G)$ by Poisson automorphisms. In particular, if $c \in \mathfrak{c}_{\CC}$ and $w \in \mathcal{W}_G$ then there is a Poisson isomorphism $w \colon \mathrm{CM}_c(G) \to \mathrm{CM}_{w(c)}(G)$. 
\end{proposition} 

\section{The classification}

Let \(G\leqslant\Sp(V)\) be a symplectic reflection group. We wish to determine the Namikawa--Weyl group \(\mathcal W_G\) associated to the linear quotient \(V/G\). By Section~\ref{sec:namikawaweyl}, this means that we need to determine the set \(\mathcal P(G)\) of conjugacy classes of minimal parabolic subgroups of \(G\) and for every such subgroup \(P\leqslant G\) the normalizer \(N_G(P)\). Finally, we need to restrict the groups \(P\) and \(N_G(P)\) to \(\Sp(V_P) \) and read off the contribution of \(P\) to \(\mathcal W_G\) from Table~\ref{tab:kleinianpairs}.

In this section, we explicitly describe the Namikawa--Weyl group for any given symplectic reflection group \(G\).
For this, we first reduce to the irreducible case.
The group \(G\) is called \emph{symplectically irreducible} if there is no decomposition \(V = V_1\oplus V_2\) with symplectic subspaces \(0\neq V_i\leqslant V\) left invariant by \(G\). If $G$ is a symplectic reflection group then it is easily seen that the action is \emph{symplectically completely reducible} in the sense that if $V_1 \leqslant V$ is a symplectic subspace left invariant under $G$ then $V = V_1 \oplus V_2$ as a $G$-module, where $V_2 = V_1^{\perp}$, and $G = G_1 \times G_2$, where $G_i$ is the image of $G$ in $\Sp(V_i)$. 

\subsection{Reduction to the irreducible case}
Assume that \(G\) is not symplectically irreducible. That is, we can write \(G = G_1\times G_2\) with \(G_i\leqslant\Sp(V_i)\) for symplectic subspaces \(0\neq V_i\leqslant V\) with \(V = V_1 \oplus V_2\). If \(P\leqslant G\) is a minimal parabolic subgroup, then by minimality \(V_P \leqslant V_i\) for exactly one \(i\in \{1, 2\}\). Without loss of generality, we assume \(i = 1\).
Hence by restricting to \(V_1\) we may identify \(P\) with a minimal parabolic subgroup of \(G_1\) and $N_G(P) = N_{G_1}(P) \times G_2$, where the morphism $N_G(P) \to \Sp(V_P)$ factors through $N_{G_1}(P)$. This implies that 
\[
(\mathcal{W}_G,\mathfrak{c}) = (\mathcal{W}_{G_1} \times \mathcal{W}_{G_2},\mathfrak{c}_1 \oplus \mathfrak{c}_2).
\]
Therefore, we may assume in the following that \(G\) is symplectically irreducible.

When \(\dim(V) = 2\), \(V/G\) is a Kleinian singularity and the Namikawa--Weyl group \(\mathcal W_G\) is immediately given by the McKay correspondence since $P = N_G(P) = G$ in this case. Therefore, we assume that \(\dim(V) \geqslant 4\).

\subsection{The classification of irreducible symplectic reflection groups}
The symplectically irreducible symplectic reflection groups were classified in the guise of quaternionic reflection groups by Cohen \cite{Coh80} with recent amendments by Taylor \cite{Tay25} and Waldron \cite{Wal25}.
As explained below, they naturally contain the irreducible complex reflection groups \cite{ST54}. We give a brief account of this classification now, sufficient for our classification problem; see the above references for more details.

Let \(G\leqslant\Sp(V)\) be a symplectically irreducible symplectic reflection group.
If $V$ is not irreducible as a complex representation of \(G\) then $V = W \oplus W^\ast$, where $W$ is a Lagrangian subspace of $V$ on which $G$ acts as an irreducible complex reflection group \(Q \leqslant \GL(W) \), and $G$ is the image of $Q$ under the diagonal embedding \(\GL(W) \to \GL(W \oplus W^\ast)\). We say that the action of $G$ is \emph{complex reducible} in this case. Thus, the classification of symplectically irreducible but complex reducible symplectic reflection groups is equivalent to the classification of irreducible complex reflection groups.

If \(G\leqslant\Sp(V)\) acts (complex) irreducibly on \(V\), then we call \(G\) \emph{symplectically imprimitive}, if there is a decomposition \(V = V_1\oplus \cdots \oplus V_k\) into symplectic subspaces \(0\neq V_i\leqslant V\) with $k \geqslant 2$ such that for every \(i\in \{1,\dots, k\}\) and every \(g\in G\) there is some \(j\in\{1,\dots, k\}\) with \(g.V_i = V_j\). Otherwise, we call \(G\) \emph{symplectically primitive}. Analogously, we call \(G\) \emph{complex imprimitive} if such a decomposition \(V = V_1\oplus \cdots \oplus V_k\) exists, where the subspaces \(V_i\) are not necessarily symplectic.
Clearly, symplectic imprimitivity implies complex imprimitivity, but a symplectically primitive group may still be complex imprimitive. Therefore, the classification of complex irreducible symplectic reflection groups is divided into three disjoint classes:
\begin{enumerate}
	\item[-] Symplectically imprimitive; 
	\item[-] Symplectically primitive but complex imprimitive; and
	\item[-] Complex primitive. 
\end{enumerate}
We consider each class separately. 

\begin{table}[h]
  \caption{Namikawa--Weyl groups for the exceptional irreducible complex reflection groups.}
  \label{tab:compexcep}
  \begin{tabular}{|l|c|}
    \(G\)      &  \(\mathcal W_G\) \\
    \hline
    \(G_4\)    & \(A_2\)   \\
    \(G_8\)    & \(A_3\)   \\
    \(G_{12}\) & \(A_1\)   \\
    \(G_{16}\) & \(A_4\)   \\
    \(G_{20}\) & \(A_2\)   \\
    \(G_{24}\) & \(A_1\)   \\
    \(G_{28}\) & \(A_1^2\) \\
    \(G_{32}\) & \(A_2\)   \\
    \(G_{36}\) & \(A_1\)
  \end{tabular}~
  \begin{tabular}{|l|c|}
    \(G\)      &  \(\mathcal W_G\) \\
    \hline
    \(G_5\)    & \(A_2^2\)                  \\
    \(G_9\)    & \(A_1\times A_3\) \\
    \(G_{13}\) & \(A_1^2\)                  \\
    \(G_{17}\) & \(A_1\times A_4\) \\
    \(G_{21}\) & \(A_1\times A_2\) \\
    \(G_{25}\) & \(A_2\)                    \\
    \(G_{29}\) & \(A_1\)                    \\
    \(G_{33}\) & \(A_1\)                    \\
    \(G_{37}\) & \(A_1\)
  \end{tabular}~
  \begin{tabular}{|l|c|}
    \(G\)      & \(\mathcal W_G\) \\
    \hline
    \(G_6\)    & \(A_1\times A_2\) \\
    \(G_{10}\) & \(A_2\times A_3\) \\
    \(G_{14}\) & \(A_1\times A_2\) \\
    \(G_{18}\) & \(A_2\times A_4\) \\
    \(G_{22}\) & \(A_1\)                    \\
    \(G_{26}\) & \(A_1\times A_2\) \\
    \(G_{30}\) & \(A_1\)                    \\
    \(G_{34}\) & \(A_1\)                    \\
               &
  \end{tabular}~
  \begin{tabular}{|l|c|}
    \(G\)      & \(\mathcal W_G\) \\
    \hline
    \(G_7\)    & \(A_1\times A_2^2\) \\
    \(G_{11}\) & \(A_1\times A_2\times A_3\) \\
    \(G_{15}\) & \(A_1^2\times A_2\) \\
    \(G_{19}\) & \(A_1\times A_2\times A_4\)\\
    \(G_{23}\) & \(A_1\) \\
    \(G_{27}\) & \(A_1\) \\
    \(G_{31}\) & \(A_1\) \\
    \(G_{35}\) & \(A_1\) \\
               &
  \end{tabular}
\end{table}

\subsection{Complex reducible groups}
Assume that \(G\) arises from a complex reflection group \(Q\) acting on a Lagrangian subspace \(W\leqslant V\).
Then minimal parabolic subgroups \(P\leqslant G\) can be identified with minimal parabolic subgroups of \(Q\) in the usual sense.
This implies that \(P\) is a cyclic group and its normalizer \(N_G(P)\) is its centralizer. In particular, the action of $\Xi_P = N_G(P) / P$ is trivial and hence \(\mathcal W_G\) is a product of symmetric groups; see \cite[Lem.~4.1]{BST18}.

If \(Q\) is an exceptional complex reflection group, then the precise description of \(\mathcal W_G\) can be found in \cite[Lem.~7.3]{BST18} and \cite[Tab.~1]{BST18} in most cases. For the remaining cases, one may find the conjugacy classes of parabolic subgroups in the tables in \cite[App.~C]{OT92}. We summarize these results in Table~\ref{tab:compexcep}.

\subsection{Symplectically imprimitive groups}
Let \(G\) be a symplectically imprimitive group. By \cite{Coh80}, the group \(G\) is conjugate to a normal subgroup of a wreath product \(K\wr \mathfrak S_n\), with \(n = \frac{1}{2}\dim(V)\) and \(K\leqslant \SL_2(\CC)\) a finite group.
When $n \geqslant 3$, all the symplectically imprimitive groups are parametrized by finite groups \(K, H\leqslant\SL_2(\CC)\) with \([K, K] \leqslant H\trianglelefteqslant K\) and written as \(G_n(K, H)\).
When \(n = 2\), there are additional groups \(G(K, H, \phi)\) with \(H\trianglelefteqslant K\) and \(\phi:K/H\to K/H\) an automorphism of order at most $2$.
See \cite{Coh80, Tay25, Wal25} for the definitions and more information on these groups.
The imprimitive complex reflection groups \(G(m, p, n)\) (in the notation of \cite{ST54}) can be included in this list as the groups \(G_n(\mathsf C_m, \mathsf C_{m/p})\).

\begin{table}[h]
  \caption{The number \(k\) of conjugacy classes of minimal parabolic subgroups \(P\) of \(G\) with \(P \cong \mathsf C_2\) acting by permutations on two 2\=/dimensional symplectic subspaces.}
  \label{tab:dim2a1}
  \def\arraystretch{1.2}
  \begin{tabular}{l|l|l}
    \(G\) & conditions & \(k\) \\
    \hline
    \(G_2(K, H)\) & \(K/H\) cyclic, \([K:H]\) even & 2 \\
    & \(K/H\) cyclic, \([K:H]\) odd & 1 \\
    \hline
    \(G_2(\mathsf D_{2d}, \mathsf C_{2d})\) & \(d\geqslant 1\) & 4 \\
    \hline
    \(G(\mathsf D_m, \mathsf C_l, \psi_r)\) & (A) and (B) & 4 \\
     & either (A) or (B) & 3 \\
     & neither (A) nor (B) & 2 \\
     \hline
    \(G(\mathsf T, \mathsf C_2, \rho(\gamma))\) && 1 \\
    \hline
    \(G(\mathsf O, \mathsf D_2, \id)\)& & 2  \\
    \hline
    \(G(\mathsf O, \mathsf C_2, \id)\) && 3 \\
    \hline
    \(G(\mathsf I, \mathsf C_2, \id)\) && 2 \\
    \hline
    \(G(\mathsf I, \mathsf C_2, \Theta)\) && 1 \\
    \hline
    \(G(\mathsf I, \{ 1 \}, \Theta)\) && 1
  \end{tabular}
  \vskip 1ex
  {\raggedright Notation:
  \begin{itemize}
    \item See \cite{Tay25} or \cite{RS26} for the definition of the maps \(\psi_r\), \(\rho(\gamma)\) and \(\Theta\).
    \item (A) is the condition \enquote{\(\frac{2m}{\gcd(2m/l, r - 1)}\) even and \(\gcd(2m/l, r - 1)\) even}.
    \item (B) is the condition \enquote{\(\frac{2m}{\gcd(2m/l, r + 1)}\) even and \(\gcd(2m/l, r + 1)\) even}.
  \end{itemize}}
\end{table}

\begin{proposition}
  \label{prop:imprim}
  Let \(G = G_n(K, H)\) or \(G = G(K, H, \phi)\).
  If \(H\neq \{1\}\), let \(W_H\) be the Weyl group corresponding to \(H\) and \(W_H^{K/H}\) the fixed group corresponding to the pair \((H, K)\) in Table~\ref{tab:kleinianpairs}.
  If \(H = \{1\}\), set \(W_H = \{1\}\).
  Then the Namikawa--Weyl group associated to \(G\) is given by
  \[
  \mathcal W_G = W_H^{K/H} \times A_1^k,
  \]
  where, for \(n \geqslant 3\), \(k = 1\) and otherwise the value of \(k\) is listed in Table~\ref{tab:dim2a1}.
\end{proposition}
\begin{proof}
  The parabolic subgroups of \(G\) and their normalizers are listed in \cite[Sect.~4]{GRS25} and \cite[Sect.~3]{RS26}, with the case of complex reflection groups covered already in \cite[Sect.~6.1]{OT92} and \cite{MT18}.
  In all cases, there are two families of minimal parabolic subgroups \(P\leqslant G\).
  The first family consists of groups \(P\) that are isomorphic (as a reflection group) to \(H\) and act non\-/trivially only on a 2\=/dimensional symplectic subspace.
  The second family consists of groups isomorphic to \(\mathsf C_2\) that permute exactly two such subspaces.
  Assuming \(H\neq \{1\}\), there is one conjugacy class of minimal parabolic subgroups in the first family.
  For \(n \geqslant 3\), there is one class of parabolic subgroups in the second family, but when $n = 2$ (so that $\dim(V) = 4$) there are in general several distinct classes, as listed in Table~\ref{tab:dim2a1}.

  As noted in the proof of Proposition~\ref{prop:SL2pairs}, the action of $N_G(\mathsf C_2)$ on $\Irr(\mathsf C_2)$ is trivial. Therefore it remains to describe the normalizer \(N_G(P)\) of a parabolic subgroup \(P\) with \(P\) acting as \(H\) on a 2\=/dimensional symplectic subspace.
  Up to conjugation, the elements of \(P\) are given by the matrices \[\begin{psmallmatrix} h & \\ & I_{2n - 2} \end{psmallmatrix}\] with \(h\in H\).
  By construction of \(G_n(K, H)\), respectively \(G(K, H, \phi)\), we must have \(N_G(P)|_{V_P} \leqslant K\).
  Because \(H\) is normal in \(K\), the group \(N_G(P)\) contains the matrices \[\begin{psmallmatrix} x & & \\ & x^{-1} & \\ & & I_{2n - 4}\end{psmallmatrix}\]
  for all $x \in K$. Or, if \(G = G(K, H, \phi)\),  we choose $y \in K$ such that $y H = \phi(x H)$ so that \[\begin{psmallmatrix} x & \\ & y\end{psmallmatrix}\] belongs to $N_G(P)$ for all \(x\in K\). In both cases, we have \(N_G(P)|_{V_P} = K\).
\end{proof}

\subsection{Symplectically primitive, complex imprimitive groups}
The references \cite{Coh80, RS26} work with quaternionic reflection groups and refer to these groups as \enquote{primitive groups with imprimitive complexification}. There are several infinite families of these groups. We briefly recall the classification from \cite[Sect.~3]{Coh80}.
For this, we require the following notation.

\begin{notation}
  For a matrix \(g\in \GL_n(\CC)\), we write \[g^\circledast = \begin{pmatrix} g & \\ & (g^\top)^{-1}\end{pmatrix}\in \Sp_{2n}(\CC).\]
  If \(T\leqslant\GL_n(\CC)\) is a group, we similarly write \(T^\circledast = \langle g^\circledast\mid g\in T\rangle \leqslant \Sp_{2n}(\CC)\).
\end{notation}

Let \(G\leqslant\Sp(V)\) be a symplectically primitive, complex imprimitive group with \(\dim(V) > 2\).
By \cite[Thm.~3.6]{Coh80}, we have \(\dim(V) = 4\) and there is a group \(T \leqslant \U_2(\CC)\) such that \(G\) is conjugate to \[\langle T^\circledast, s\rangle,
\] where \[s = \begin{pmatrix} & & & 1\\ & & -1 & \\ & -1 & & \\ 1 & & & \end{pmatrix}.\]
For \(d\geqslant 1\), let \(\mu_d \leqslant \CC^\times\) be the group of all \(d\)-th roots of unity.
Then the group \(T\) is given as \(T = \mu_d T_0\), where \(T_0\leqslant \U_2(\CC)\) is a primitive complex reflection group and \(d\) is a multiple of \(|{Z(T_0)}|\).
By \cite[Lem.~3.3]{Coh80}, the group \(T_0\) can be one of the exceptional complex reflection groups \[G_5, G_7, G_8,\dots, G_{22}.\]
See \cite[Lem.~3.3]{Coh80} and \cite[Thm.~7.9]{Tay25} for the possible values of \(d\) for each given \(T_0\).

\begin{proposition}
  \label{prop:primimprim}
  Let \(G = \langle (\mu_d T_0)^\circledast, s\rangle\leqslant\Sp(V)\) be a symplectically primitive, complex imprimitive group.
  Then the Namikawa--Weyl group associated to \(G\) depends only on \(T_0\), but not on \(d\).
  The resulting groups are listed in Table~\ref{tab:primimprim}.
\end{proposition}
\begin{proof}
  Let \(G = \langle T^\circledast, s\rangle\) with \(T = \mu_d T_0\).
  Because \(\dim(V) = 4\), all proper non\-/trivial parabolic subgroups of \(G\) are minimal.
  Hence the conjugacy classes of minimal parabolic subgroups of \(G\) are listed in \cite[Table~3]{RS26}.
  In particular, \cite[Table~3]{RS26} implies that the minimal parabolic subgroups are cyclic and the number of conjugacy classes does not depend on \(d\).

  Let \(P\leqslant G\) be a minimal parabolic subgroup.
  Then \(P \cong \mathsf C_2\) or \(P = Q^\circledast\) for a (complex) parabolic subgroup \(Q\leqslant T_0\).
  In the first case, there is nothing further to compute. 
  In the second case, the normalizer \(N_G(P)\) of \(P\) is computed in \cite[Prop.~4.12]{RS26}.
  If \(P \cong \mathsf C_3\) and \(T_0\in \{G_5, G_7\}\), then \(N_G(P) = P\rtimes Z(T)^\circledast\).
  Then such a group \(P\) contributes a group of type \(A_2\) to the Namikawa--Weyl group.

  In the remaining cases, the proof of \cite[Prop.~4.12]{RS26} shows that \(N_G(P) = P\rtimes \langle C^\circledast, g^\circledast s\rangle\), where \(C\leqslant T\) is the complement of \(Q\) in \(N_T(Q)\) and \(g\in T\).
  To finish the proof it remains to show that \(g^\circledast s\) acts non-trivially on \(P\) because this implies that the image of the pair \((P, N_G(P))\) in $\Sp(V_P)$ is isomorphic to \((\mathsf C_k, \mathsf D_r)\) with \(k \mid 2r\) by Table~\ref{tab:kleinianpairs}.
  We may assume that \(|{P}| > 2\).

  To avoid excessive use of the \(\circledast\) operator, we consider \(G\) and \(P\) as quaternionic reflection groups with \(G\leqslant \GL_2(\HH)\) for the remainder of the proof.
  Let \(s' = \begin{psmallmatrix} 0 & -\mathbf j\\ \mathbf j & 0\end{psmallmatrix}\in \GL_2(\HH)\) with \(\mathbf j\in \HH^\times\) a root of \(-1\) such that \(\HH = \CC \oplus \CC\mathbf j\).
  The matrix \(s\) is the \enquote{complexification} of \(s'\) in the sense of \cite[p.~294]{Coh80} and so \(G\) is given by \(\langle T, s'\rangle\) as a quaternionic reflection group, where we consider \(T\) as a subgroup of \(\GL_2(\HH)\) in the natural way.
  Recall that \(P\), and hence \(Q\), is cyclic and let \(q\in Q\) be a generator.
  Because \(Q\) is a complex reflection group, \(q\) is a reflection and \(1 \neq \det(q)\in \CC^\times\) is a root of unity.
  Hence, by \cite[Lem.~4.5]{RS26}, we have \(s'q s' = \det(q)^{-1} q\).
  So we have \((gs') q (gs')^{-1} = \det(q)^{-1}(gqg^{-1})\).
  We conclude \(\det((gs')q(gs')^{-1}) = \det(q)^{-1}\) and we have \(\det(q)\notin\{\pm 1\}\) by the assumption \(|{P}| > 2\).
  Therefore \(q\neq(gs') q (gs')^{-1}\) and \(N_G(P)\) does not act as a cyclic group on \(V_P\).
  Finally, we see from Table~\ref{tab:kleinianpairs} that the contribution of \(P\) to the Namikawa--Weyl group is of type \(B\).
\end{proof}

\begin{table}
  \caption{Namikawa--Weyl groups for the symplectically primitive, complex imprimitive groups \(G = \langle (\mu_d T_0)^\circledast, s\rangle\).}
  \label{tab:primimprim}
  \begin{tabular}{|l|c|}
  \(T_0\)                & \(\mathcal W_G\) \\
  \hline
  \(G_5\)                & \(A_1^2 \times A_2\)\\
  \(G_7\)                & \(A_1^3 \times A_2\)  \\
  \(G_8\)                & \(A_1 \times B_2\)  \\
  \(G_9\)                & \(A_1^3 \times B_2\)  \\
  \(G_{10}\)             & \(A_1^2 \times B_2\)  \\
  \(G_{11}\)             & \(A_1^4 \times B_2\)  \\
  \(G_{12}\)             & \(A_1^2\)  \\
  \(G_{13}\)             & \(A_1^4\)  \\
  \(G_{14}\)             & \(A_1^3\)  \\
  \end{tabular}~
  \begin{tabular}{|l|c|}
  \(T_0\)                & \(\mathcal W_G\) \\
  \hline
  \(G_{15}\)             & \(A_1^5\)  \\
  \(G_{16}\)             & \(A_1^2 \times B_2\)  \\
  \(G_{17}\)             & \(A_1^3 \times B_2\)  \\
  \(G_{18}\)             & \(A_1^3 \times B_2\)  \\
  \(G_{19}\)             & \(A_1^4 \times B_2\)  \\
  \(G_{20}\)             & \(A_1^3\)  \\
  \(G_{21}\)             & \(A_1^4\)  \\
  \(G_{22}\)             & \(A_1^3\) \\
   &
  \end{tabular}
\end{table}

\begin{table}
  \caption{Namikawa--Weyl groups for the symplectically primitive, complex primitive groups \(G = W(X)\) with \(X\) a root system from \cite[Table~II]{Coh80}.}
  \label{tab:excep}
  \begin{tabular}{|l|c|}
    \(G\) & \(\mathcal W_G\) \\
    \hline
    \(W(O_1)\) & \(A_1\) \\
    \(W(O_2)\) & \(A_1\) \\
    \(W(O_3)\) & \(A_1^2\) \\
    \(W(P_1)\) & \(B_2\) \\
    \(W(P_2)\) & \(G_2\) \\
    \(W(P_3)\) & \(A_1 \times G_2\) \\
    &
  \end{tabular}~
  \begin{tabular}{|l|c|}
    \(G\) & \(\mathcal W_G\) \\
    \hline
    \(W(Q)\) & \(A_1\) \\
    \(W(R)\) & \(A_1\) \\
    \(W(S_1)\) & \(A_1\) \\
    \(W(S_2)\) & \(A_1\) \\
    \(W(S_3)\) & \(A_1\) \\
    \(W(T)\) & \(A_1\) \\
    \(W(U)\) & \(A_1\)
  \end{tabular}
\end{table}

\subsection{Symplectically primitive, complex primitive groups}
Let \(G\leqslant \Sp(V)\) be a symplectically primitive, complex primitive group with \(\dim(V) > 2\).
Then \(G\) is conjugate to one of 13 groups with root systems listed in \cite[Table~II]{Coh80}.
The minimal parabolic subgroups of these groups and their normalizers can be found in the tables in \cite[Sect.~5.2]{RS26}.
With this information one immediately constructs the Namikawa--Weyl groups using Table~\ref{tab:kleinianpairs}. The results are summarized in Table~\ref{tab:excep}.

\bibliographystyle{amsalpha}
\providecommand{\bysame}{\leavevmode\hbox to3em{\hrulefill}\thinspace}
\providecommand{\MR}{\relax\ifhmode\unskip\space\fi MR }
\providecommand{\MRhref}[2]{%
	\href{http://www.ams.org/mathscinet-getitem?mr=#1}{#2}
}
\providecommand{\href}[2]{#2}


\end{document}